\documentclass[12pt]{amsart}
\usepackage{latexsym,amscd,amssymb,epsfig}
\usepackage[dvips]{color}

\theoremstyle{plain}
\newtheorem{thm}{Theorem}
\newtheorem{lemma}[thm]{Lemma}
\newtheorem{prop}[thm]{Proposition}

\newtheorem*{thm*}{Theorem}
\newtheorem*{lemma*}{Lemma}
\newtheorem*{prop*}{Proposition}
\newtheorem*{cor*}{Corollary}
\newtheorem*{conj*}{Conjecture}

\theoremstyle{definition}

\theoremstyle{remark}
\newtheorem{rmk}{Remark}
\newtheorem{obs}{Observation}

\newcommand{\ind}{\mbox{$\perp \kern-5.5pt \perp$}}

\begin{document}

\title[Shellability characterization of geometric lattices]{A lexicographic shellability 
characterization of geometric lattices}
\thanks{The first author was supported by NSF grant DMS-0954865.  
The second author was supported by NSF grant DMS-1002636, DMS-1200730 and
the Ruth I. Michler Prize of the Association for Women in Mathematics.
}
\keywords{Geometric lattice \and M\"obius function \and Shellability \and Order complex}
 \subjclass{05B35 , 52B22 , 06A07 , 05E45}


\author{Ruth Davidson}
\address{North Carolina State University}
\email{redavids@ncsu.edu}

\author{Patricia Hersh}
\address{North Carolina State University}
\email{plhersh@ncsu.edu}

\maketitle

\begin{abstract}
Geometric lattices are characterized in this paper as those finite, atomic lattices  such that every
atom ordering induces a lexicographic shelling given by an edge 
labeling known as a minimal labeling.  Equivalently, geometric lattices are shown to 
be exactly those 
finite lattices such that every ordering on the join-irreducibles induces a lexicographic
shelling.
This  new characterization fits into a similar paradigm as  McNamara's  characterization of supersolvable lattices as those lattices admitting a different type of lexicographic shelling, 
namely one in which each
maximal chain is labeled with a permutation of $\{ 1,\dots ,n \} $.   
\end{abstract}

\section{Introduction}
\label{intro}
 In \cite{McNamara}, 
 McNamara  proved that supersolvable lattices can be characterized as lattices admitting a certain type of EL-labeling known as an $S_n$-EL-labeling.   Each maximal chain is labeled by the set of labels $\{ 1,\dots ,n\} $ with each label used exactly once in each maximal chain.  Previously, Stanley had proven that all  supersolvable lattices admit such EL-labelings in \cite{Stanley72}.
 Thus, McNamara's result gave a new characterization of supersolvable lattices: that a 
 finite lattice is supersolvable if and only if it has an $S_n$-EL-labeling.
 
This paper gives a result of a similar spirit for
 geometric lattices -- 
a new characterization of geometric lattices as the lattices admitting a family of 
lexicographic shellings induced by the various possible orderings on the join-irreducibles.  
The join-irreducibles turn out to be exactly the atoms in this case.
Geometric lattices are well-known to have the property that every atom ordering induces an EL-labeling
by labeling each cover relation $u\lessdot v$ with the smallest atom that is below $v$ but not
$u$.  
We will prove the converse: that all finite atomic lattices in which every atom ordering induces an EL-labeling are geometric lattices.  We also prove the following 
reformulation which was suggested to us 
by Peter McNamara as a way of removing the hypothesis that our lattices are atomic: 
that a finite lattice $L$ is geometric if and only if
every ordering of the join-irreducibles induces a minimal labeling that is 
an EL-labeling. 

There is an extensive literature on the notion of lexicographic shellability. 
Most of the emphasis is on proving that important 
classes of partially ordered sets admit 
lexicographic shellings.   
For example, upper semimodular lattices (
\cite{Ga}), geometric lattices
(
\cite{Stanley74},  
\cite{bjorner80}), geometric  semilattices (
\cite{WW}),  supersolvable lattices (
\cite{Stanley72}), subgroup lattices of 
solvable groups (
\cite{Sh}, 
\cite{Wo}), and 
Bruhat order (
\cite{bjorner82}) are all known to be 
lexicographically shellable.
Our point is to take things in the opposite direction, namely to use the types of
lexicographic shellings (induced by EL-labelings) that
geometric lattices are known to have as a way of characterizing geometric lattices. 

A motivation for this work is that it gives a new working definition for geometric lattices.   Since geometric lattices are exactly the lattices of flats of matroids and contain enough data to 
determine the matroids, this also gives a new characterization for matroids.   The condition that each ordering on the atoms of a finite atomic lattice  induces an EL-labeling seems like something one might  sometimes be able to prove for posets of interest  without  knowing a priori that one has a geometric lattice.  On the other hand, this is a strong enough condition that we had wondered whether this was a strictly larger class of posets than the geometric lattices, thinking that an affirmative answer could lead to the discovery of some very interesting new posets.  
What led us to this project 
was a result of \cite{nyman} where it was natural to ask if the result could be extended to those 
finite atomic lattices in which every atom ordering admits an EL-labeling;  we quickly  realized 
that  we 
did not know whether this latter class was actually larger than the class of geometric lattices.
Theorem ~\ref{atom-theorem} gives a negative answer to this question.

Our method of proof also establishes a direct link between two seemingly quite distinct types of poset theoretic structure.  That is, we give a construction which takes as its input any pair of lattice elements $x,y$ which together prevent the  lattice from being semimodular and which produces from this  a nonempty family of orderings on the join-irreducibles whose associated edge labelings do not give  lexicographic shellings.    In other words, we take the local structure of having $x,y$ which both cover $x\wedge y$ but are not both covered by $x\vee y$, and we translate that into a combinatorial-topological structure precluding certain orderings on the maximal chains from yielding shellings of the simplicial complex typically associated to a poset, the order complex of the poset.

There is reason to be especially interested in understanding the lexicographic shellings for
geometric lattices,  as we now describe. 
The class of  
geometric lattices  
includes all of the intersection lattices of real, central  hyperplane arrangements.
Zaslavsky expressed the number of regions in the complement of a real hyperplane arrangement in terms of M\"{o}bius functions of geometric lattices and geometric semi-lattices in \cite{Zaslavsky}.  
Lexicographic shellability is one of the most powerful available methods for determining 
M\"obius functions, and the known lexicographic shellings for geometric lattices indeed yield 
M\"obius function formulas in terms of the 
so-called nbc-bases of the matroid with respect to a choice of ordering on the ground set of the 
matroid.
Thus, lexicographic shellability  
in the case of geometric lattices gives a method for region counting in the complement of a real hyperplane arrangement; the existence of many different lexicographic  shellings, namely one for each atom ordering,  also has the ramification of giving numerous distinct but equal expressions for the number of regions in a real, central hyperplane arrangement.  Thus, it seemed interesting to know as much as possible about the lexicographic shellings of this  class of posets and potentially also for closely related classes of posets.

McNamara's characterization of 
supersolvable lattices has given a useful new way of proving classes
of  lattices to be supersolvable.
See e.g. \cite{BC} for one such result.  
Our  results will imply there is a similar potential 
for geometric lattices.  

\section{Background and terminology}

Let $P$ be a finite poset.  Let $E(P)$ denote the set of
edges of the Hasse diagram of $P$.  We write $x \lessdot y$
to indicate that $y$ covers $x$ in $P$, namely $x\le z\le y$ implies $z=x$ or $z=y$.   If $\lambda : E(P) \to \mathbb{N}$ is an edge labeling of the Hasse diagram of $P$ and $x \lessdot y$, then we write $\lambda(x,y)$ to indicate the label given to the edge from $x$ to $y$.  Recall that $\lambda$ is an \textit{EL-labeling} for $P$ if for every interval $[x, y] $ of $ P$, there is a unique \textit{rising chain} $C:  = x  = x_{1} \lessdot x_{2} \lessdot \cdots \lessdot x_{j} = y$ where $\lambda(x,x_{2}) \leq \lambda(x_{2}, x_{3}) \leq \cdots \leq \lambda(x_{j-1}, y)$, and the label sequence of $C$ is lexicographically smaller than the label sequence of every other saturated
chain in the interval $[x,y]$ (cf.  \cite{bjorner80}).  
It is well-known that an EL-labeling gives a shelling order for the facets (maximal faces)
of the order complex $\Delta (P)$ of $P$.  

We now review the notion of geometric lattice as well as the types of  EL-labelings which 
they are already known to possess.  An 
{\it atom } in a poset $P$ with unique  minimal element $\hat{0} $ is any $a\in P$ which
covers
$\hat{0} $.  A finite poset $P$ is {\it graded } 
if every maximal chain has the same number of elements
in it; in this case, there is a rank function $\rho $ defined recursively
by $\rho (x) = 0$ for $x$  a minimal 
element of $P$ and $\rho (y) = \rho (x) + 1$ for $x\lessdot y$.
A {\it lattice } is a poset such that any pair of elements $x,y$ has a unique
least upper bound $x\vee y$ and a unique greatest lower bound $x\wedge y$.
A lattice is \textit{atomic}  if 
every element is a join of atoms.  A lattice is \textit{semimodular} if it has
a rank function $\rho$ that satisfies
$$ (i) \ \rho(x \wedge y) + \rho(x \vee y) \leq \rho(x) + \rho(y). $$
A finite lattice is \textit{ geometric} if it is atomic and semimodular.  

Our interest is in using the existence of a certain family of 
edge labelings for a poset $P$ 
to show that $P$ fits into the class of  
geometric lattices.   Therefore, let us now introduce the labelings we
will use.
%

An element $x$ in a lattice $L$ is  a {\it join-irreducible} if $x= y\vee z$ implies $x=y$ or $x=z$.
Let $\mbox{JoinIrred(L)}$ denote the set of join-irreducibles of $L$.  For $x \in L$, let 
$$
\mbox{JoinIrred}(x) = \{w \leq x | w \in \mbox{JoinIrred}( L)  \}.  $$
Let $\mathcal{A}(L)$ denote the atoms of $L$, and for $x \in L$ let $A(x) = \{ a \leq x | a \in \mathcal{A}(L) \}$.  Note that in an atomic lattice $L$ we have $\mbox{JoinIrred}(L) = \mathcal{A}(L)$. 

Let $n = |\mbox{JoinIrred}(L)|$. 
Given any bijection $\gamma : \mbox{JoinIrred}(L) \to [n]$, 
the map $\gamma$ induces a \textit{minimal labeling} $\lambda_{\gamma} :  E(L) \to [n]$ 
by the rule 
$$\lambda_{\gamma}(x, y) = \min \{\gamma (j) |  \ j \in \mbox{JoinIrred}(y) \setminus 
\mbox{JoinIrred}(x) \} .$$ 
\par

\begin{thm}[Bj\"orner \cite{bjorner80}]
\label{bjorner-theorem}
The minimal labeling resulting from any linear  ordering of the atoms in a geometric lattice is
an EL-labeling.
\end{thm}

The following proposition, which appears as Corollary 1, p. 81, in \cite{Birkhoff},  gives an
alternate formulation of semimodularity for graded lattices that will be convenient in our proofs.

\begin{prop}[Birkhoff] \label{equiv-prop}
Let $L$ be a finite lattice.  The following two conditions are equivalent: 
 \begin{itemize}
\item $L$ is graded, and the rank function $\rho$ of $L$ satisfies  the semimodularity 
condition (i) above.
\item If $x$ and $y$ both cover $x \wedge y$, then $x \vee y$ covers both $x$ and $y$.  
\end{itemize} 
\end{prop}

See e.g. \cite{EC} for further background on posets.




\section{Lexicographic shellability characterizations for geometric lattices}

This section is devoted to 
giving two new characterizations of geometric lattices.  The first is based on atom orderings for 
finite atomic lattices.  
The second characterization  replaces this with  a condition on the  
orderings of the  join-irreducibles for finite lattices  so as to 
avoid assuming a priori that 
the lattices are atomic.
In both cases, we prove for any finite lattice $L$ that if  
every ordering of the join-irreducibles 
induces a  minimal labeling which is an EL-labeling, then $L$ is a  geometric lattice. 
To this end, we
first develop some helpful properties of minimal labelings.  

\begin{lemma}\label{distinct-lemma}
Let $L$ be a finite atomic lattice and let $\lambda_{\gamma}$ be a minimal labeling  on $L$.
 Then $x_i \lessdot x_{i+1}\le x_j \lessdot x_{j+1}$ in $L$ implies 
 $\lambda _{\gamma}(x_{i}, x_{i + 1}) \neq \lambda_{\gamma}(x_{j}, x_{j+1})$.  
  In other words, the labels 
on any particular saturated chain are distinct.
\end{lemma}

\begin{proof}
This is immediate from the fact  that $A(x_{j+1} ) \setminus A(x_j)$ is disjoint
from $A(x_{i+1})\setminus A(x_i)$ for $i\ne j$.
\end{proof}

\begin{lemma}\label{atom-contain}
Let $L$ be a finite 
lattice.  Then $\mbox{JoinIrred}(u)\subseteq \mbox{JoinIrred}(v)$ 
if and only if
$u \leq v$.  Moreover,  $u=v$ if and only if  $\mbox{JoinIrred}(u)=\mbox{JoinIrred}(v)$. In the special case of 
a finite atomic lattice
$L$ we have $A(u)\subseteq A(v)$ if and only if $u\le v$, and we have $A(u)=A(v)$ if and only
if $u=v$.
\end{lemma}

\begin{proof}
This  follows from two facts: (1) that every element of a finite lattice $L$ is a join of join-irreducibles,  
and (2) that the only join-irreducibles in an atomic lattice are the atoms.
\end{proof}


\begin{lemma}\label{cover-x}
Let $L$ be a finite lattice and suppose that there exist  $x, y \in L$ that both cover $x \wedge y$, but with $x$ not  covered by $x \vee y$.  
 Then for $j$ any join-irreducible satisfying $y = (x \wedge y) \vee j$, we have that  
 $j \notin \mbox{JoinIrred}(z)$ for any $z$ covering $x$. 
\end{lemma}

\begin{proof}
Assume by way of contradiction that the join-irreducible $j$ given above satisfies 
$j \leq z$ for some  $x \lessdot z$.
Note that $x\wedge y \le x \lessdot z$, which together with $j\le z $ implies 
$( x \wedge y) \vee j \leq z$.
But $ (x\wedge y) \vee j = y$, so we may conclude  that $x \vee y \leq z$.  This contradicts 
the fact  that  $x\vee y$ does not cover $x$, completing our proof.
\end{proof}





Now to our first characterization of geometric lattices.

\begin{thm}\label{atom-theorem}
Let $L$ be a finite atomic lattice.  Then $L$ is geometric if and only if 
every atom ordering induces a minimal labeling that is an EL-labeling. 
\end{thm}

\begin{proof}


Bj\"orner proved in Theorem ~\ref{bjorner-theorem} that all of the minimal labelings for a geometric 
lattice are EL-labelings.  We now prove the converse.
Since we assume that  $L$ is atomic, what remains is  to prove that $L$ is semimodular.

Suppose otherwise. 
By Proposition \ref{equiv-prop},  
there must exist $x, y  \in L$ such that $x$ and $y$ both cover $x \wedge y$ but $x \vee y$ does not cover  $x$.
By Lemma ~\ref{atom-contain}, we may choose some atom $a_{x} \in A(x) \setminus A(x \wedge y)$ such that $a_x\not\in A(y)$.  Since $L$ is an atomic lattice, there must  also exist $a_{y} \in A(y)$ such that $(x \wedge y) \vee a_{y} = y$.  By Lemma ~\ref{cover-x},
$a_{y} \notin A(z)$ for any $z$ such that $x \lessdot z$.  This implies
$a_{y} \neq a_{x}$, since $a_y \not\in A(z)$ for all $z$ satisfying
$x\lessdot z$, which in particular  implies $a_y\not\in A(x)$.

Now consider any atom ordering  $\gamma : \mathcal{A}(L) \to [n]$ 
such that $\gamma(a_{x}) = 1$ and $\gamma(a_{y}) = 2$.  
%
%
Since $a_{x} \in A(x) \setminus A(x \wedge y)$ and $\gamma(a_{x}) = 1$, we know that 
$\lambda_{\gamma}(x \wedge y, x) = 1$. 
Let 
$$
C : = \quad x\wedge y = x_0 \lessdot x = x_{1} \lessdot x_{2} \lessdot \cdots \lessdot x_{k} = x \vee y
$$
be the lexicographically smallest saturated chain on the interval $[x\wedge y, x\vee y]$.
By Lemma ~\ref{cover-x}, 
$a_{y} \notin A(x_{2})$. 
Therefore, 
$\lambda_{\gamma}(x_1, x_2) \neq 2$.  By Lemma ~\ref{distinct-lemma}, 
there is no repetition in the label sequence, implying $\lambda_{\gamma}(x_1, x_2) > 2$. 
For some $2 < j \leq k$, we must have $a_{y} \in A(x_{j}) \setminus A(x_{j-1})$, implying 
$\lambda_{\gamma }(x_{j-1},x_j) = 2$.

But $\min \{\gamma(a) | a \in A(x_{2}) \setminus A(x_{1} ) \} \geq 3$, so $\lambda_{\gamma}(x_1, x_{2}) > \lambda_{\gamma}(x_{j-1}, x_{j})$. 
This implies that  $C$ 
cannot have weakly increasing labels, hence that 
 $\lambda_{\gamma}$ is not an EL-labeling.  
 \end{proof}
 

Next we give a closely related alternative 
characterization of geometric lattices which avoids making
the assumption that 
the lattices are atomic.  The essence of the proof
will be a reduction to the atomic case.


\begin{thm}\label{JI-theorem}
A finite lattice $L$ is a geometric lattice if 
and only if  every ordering of the join-irreducibles 
induces a minimal labeling $\lambda_{\gamma}$ which is an EL-labeling.
\end{thm}

\begin{proof}
One direction is well-known, so we focus on the other direction.  That is, we will assume there is 
some join-irreducible that is not an atom, and use this to produce an ordering on join-irreducibles whose associated minimal labeling is not an EL-labeling.  The case where all join-irreducibles are atoms has already been handled in Theorem ~\ref{atom-theorem}.
 
Suppose there exists $v \in \mbox{JoinIrred}(L)$  that is not an atom.  In this case, we may  choose such a $v$ so that if $a  <  v$ for  $a \in \mbox{JoinIrred}(L)$ then $a$ is an atom.  It is well known (see \cite{EC},  p. 286) that in a finite lattice, the join-irreducibles are precisely the elements that cover exactly one other element.  Let $u$  be the unique element in $L$ with $u \lessdot v$.  Thus,  $\mbox{JoinIrred}(u)$ is entirely composed of atoms and $|\mbox{JoinIrred}(u)| = k$ for some $1 \leq k \leq n-1$ where $n$ is the number of join-irreducibles in $L$.  Consider an
ordering $\gamma$ on the join-irreducibles such 
that $\gamma(v) = 1$ and  $ \{ \gamma(x) | x\in \mbox{JoinIrred}(u) \} 
= \{2, 3, \dots , k + 1\}$.    
The lexicographically smallest label sequence for any  saturated 
chain in the interval $[\hat{0}, v]$ must then have a descent,  because all saturated chains must include $u$, but  $\lambda_{\gamma}(u,v) = 1$ while  $\lambda_{\gamma}(x,y) > 1$ for all  covering relations $x \lessdot y$ in the interval $[\hat{0}, u]$.  Thus,  this minimal labeling $\lambda_{\gamma }$ 
is not an EL-labeling.
\end{proof}

\section{Lexicographic shellability characterization of semimodular lattices}

We recently learned that our next theorem  previously appeared in \cite{Ri}.  We nonetheless include our proof of Theorem ~\ref{semimodular} 
both for its new approach to this sort of question and because it 
provides a significantly higher level of detail than appears in the argument 
in  \cite{Ri}.  First we make an observation that the proof will rely upon.

\begin{obs}\label{extend}
Let $L$ be a finite lattice with $|\mbox{JoinIrred}(L)| = n$ and let $x \in L$.  If $|\mbox{JoinIrred}(x)| = k$ and $\hat{\gamma} : \mbox{JoinIrred}([\hat{0}, x]) \to [k]$ is a linear extension of the subposet of join-irreducibles of the interval $[\hat{0}, x]$, then there exists a linear extension $\gamma: \mbox{JoinIrred}(L) \to [n]$ of the subposet $\mbox{JoinIrred}(L)$ that restricts to the map $\hat{\gamma}$.  
\end{obs}

\begin{thm}\label{semimodular}

Let $L$ be a finite lattice with  $|\mbox{JoinIrred}(L)| = n$.  Suppose that
for every linear extension $\gamma : \mbox{JoinIrred}(L) \to [n]$ of the subposet  $\mbox{JoinIrred}(L)$ of join-irreducibles in $L$, the resulting 
minimal labeling $\lambda_{\gamma}$ is an EL-labeling on $L$.  Then $L$ is (upper) semimodular.  

\end{thm}

\begin{proof}

Assume by way of contradiction that $x$ and $y$ cover $x \wedge y$ but that $x \vee y$ does not cover $x$.  
Lemma \ref{atom-contain} shows that 
there exist  join-irreducibles $j_{x} \in \mbox{JoinIrred}(x) \setminus \mbox{JoinIrred}(x \wedge y)$ and $j_{y} \in \mbox{JoinIrred}(y) \setminus \mbox{JoinIrred}(x \wedge y)$.  
Let $k$ be 
the number of elements in 
$\mbox{JoinIrred}(x \wedge y)$.  
Notice that if  $k = 0$, then $x$ and $y$ are atoms, so in particular $x$ and $y$ are
join-irreducibles $j_x := x $ and $j_y := y$.  

Now we choose a linear extension $\gamma$ of the subposet $\mbox{JoinIrred}(L)$ of $L$ comprised of the
join-irreducibles.  By Observation \ref{extend},  we may  choose $\gamma $ so  that it assigns
exactly the values in $\{ 1,\dots ,k \} $  to 
the 
 join-irreducibles in $[\hat{0}, x \wedge y]$.  
Moreover, we may insist that  $\gamma(j_{x}) = k + 1$ and $\gamma(j_{y}) = k + 2$,   
choosing how $\gamma $ assigns  the values in $\{ k+3,\dots ,n\} $
to the subposet comprised of the 
remaining  join-irreducibles by taking any linear extension of the remaining join-irreducibles.  

Denote  the lexicographically smallest  maximal 
chain in the interval $[x \wedge y, x \vee y]$ by 
$$C = x \wedge y  \lessdot x \lessdot x_{2}\lessdot \cdots x_{m-1} \lessdot x_{m} = x \vee y .$$
We have assumed that $x \vee y$ does not cover $x$, implying $m>2$.   Our constraints given
above on 
our choice of $\gamma $  imply that $\lambda_{\gamma }(x_1,x_2) \not\in \{ 1,\dots ,k+2 \} $, since 
Lemma   \ref{cover-x} ensures that $j_y\not\in \mbox{JoinIrred}(x_2)$.   Thus,
$\lambda_{\gamma}(x, x_{2})  \geq k + 3$. 
But since $\mbox{JoinIrred}(y) \subset \mbox{JoinIrred}(x \vee y)$, we then must have  $j_{y} \in \mbox{JoinIrred}(x_{\ell}) \setminus \mbox{JoinIrred}(x_{\ell-1})$ for some $2 < \ell \leq m$. This implies   $\min(\{ \gamma(j) | j \in \mbox{JoinIrred}(x_{\ell}) \setminus \mbox{JoinIrred}(x_{\ell -1})\}  )  \leq k + 2$.  Hence,  $\lambda_{\gamma}(x_{\ell -1}, x_{\ell}) \leq k + 2 < r =  \lambda_{\gamma}(x,x_{2}) $, forcing the chain $C$ to  have a descent, a contradiction to this being an
 EL-labeling.
 Thus, knowing that $x$ and $y$ both cover $x\wedge y$  does  imply in our setting that
  $x\vee y$ covers both 
 $x$ and $y$.  Thus,   $L$ is semimodular.
\end{proof}

\medskip

\section{Concluding remarks}


\begin{rmk}
Axel Hultman has  informed us (personal communication) that the proof of
Theorem  ~\ref{atom-theorem}  
may also easily be modified to yield the 
following statement: Let $L$ be a finite lattice with set of join-irreducibles $\mbox{JoinIrred}(L)$ and
$k = |\mbox{JoinIrred}(L)|$.  Then 
the  labeling $\lambda_{\gamma }$ induced by 
each choice of order-preserving bijection $\gamma : \mbox{JoinIrred}(L) \rightarrow [k]$
is an  R-labeling if and
only if $L$ is  semimodular.
\end{rmk}

 If $M = M(S)$ is a matroid of rank $r$ on a finite set $S$, recall that the  \textit{independence complex} of $M$ is the $(r-1)$-dimensional simplicial complex formed by the family of all independent sets in $M$.   On the other hand, a geometric lattice is the lattice of \textit{flats}, or closed sets, of a matroid.  The matroid structure of a geometric lattice $L$ is the matroid with ground set $\mathcal{A}(L)$ and the closure operator $cl(W)$ on a subset $W \subseteq \mathcal{A}(L)$ is $cl(W) =  \{ a \in A(x)  |  a\le \vee_{w\in W} w \} $.

\begin{rmk}
There is a well-known result concerning matroid complexes which has a similar flavor to Theorem \ref{atom-theorem}.  This appears e.g. in  \cite{bjorner90} as Theorem 7.3.4.  The statement of this result is as follows:  a simplicial complex $\Delta $ is the independence complex of a  matroid  if and only if $\Delta $ is pure and every ordering of the vertices induces a shelling of $\Delta $.  
It seems  
interesting to note the resemblance between the necessary hypotheses for Theorem 7.3.4 of
\cite{bjorner90} and those of 
our characterization(s) of geometric lattices.  
It is natural to ask if one result is a translation of the other into a different 
language.  This does not appear to be the case, rather 
the two results seem to be 
fundamentally quite distinct. 
\end{rmk}

\section{Acknowledgments}

\medskip
The authors thank Peter McNamara for pointing out that one does not need to assume that a finite lattice is atomic in order 
to give a lexicographic shellability characterization of geometric lattices.  They thank Axel Hultman for suggesting another consequence of our work regarding the relationship between R-labelings and semimodular lattices.  They thank Vic Reiner for suggesting the question of characterizing semimodular lattices using many of the same  techniques as are applied to geometric lattices.  Finally, the authors thank the anonymous referees for very helpful feedback on earlier versions of the paper.

\medskip


\begin{thebibliography}{}


\bibitem{BC} R.  Biagioli and F.  Chapoton, Supersolvable LL-lattices of
binary trees, Discrete Math, {\bf 296} (2005), no. 1, 1--13.

\bibitem{Birkhoff} G. Birkhoff, Lattice Theory,  Third edition,  American Mathematical
Society Colloquium Publications, Vol. XXV, {\it American Mathematical Society, Providence,
R.I.} 1967, vi + 418 pp.

\bibitem{bjorner80} A. Bj\"{o}rner,  Shellable and Cohen-Macaulay partially ordered sets, \textit{Transactions of the American Mathematical Society} \textbf{260} (1980), 159-183.


\bibitem{bjorner90} A. Bj\"{o}rner,  Homology and Shellability of Matroids and Geometric Lattices,   \textit{Matroid Applications},  Editor: Ned White,  Cambridge University Press, 1992.  


\bibitem{bjorner82} A. Bj\"{o}rner and M. Wachs,  Bruhat order of Coxeter groups and Shellability,  \textit{Advances in Mathematics} \textbf{43} (1982),  87--100.  


\bibitem{Ga}
A.  Garsia, Combinatorial methods in the theory of Cohen-Macaulay rings, 
{\it Adv. in Math.} {\bf 38} (1980), no. 3, 229--266.

\bibitem{McNamara} P. McNamara,  EL-labelings, supersolvability, and 0-Hecke algebra actions on posets,  \textit{Journal of Combinatorial Theory Series A} \textbf{101} (2003), 69--89.   


\bibitem{nyman} K.  Nyman and E. Swartz.  Inequalities for the $h$-Vectors and Flag $h$-Vectors of Geometric Lattices,  \textit{Discrete and Computational Geometry} \textbf{32} (2004),  533--548.  

\bibitem{Ri}
I. Rival, A note on linear extensions of irreducible elements in a finite lattice, {\it Alg. Universalis}
{\bf 6} (1976), 99--103.

\bibitem{Sh} J. Shareshian, On the shellability of the order complex of the subgroup lattice of
a finite group, {\it Trans. Amer. Math. Soc.} {\bf 353} (2001), no. 7, 2689--2703.

\bibitem{Stanley72} R. Stanley,  Supersolvable lattices,  \textit{Algebra Universalis} \textbf{2}  (1972), 197--217.

\bibitem{Stanley74} R.  Stanley, Finite lattices and Jordan-H\"older sets, {\it Algebra
Universalis} \textbf{4} (1974), 361--371.



\bibitem{EC} R. Stanley,  \textit{Enumerative Combinatorics Volume I}, Cambridge University Press, 1997.  

\bibitem{WW} M. Wachs and J. Walker,  On geometric semilattices,
{\it Order} {\bf 2} (1986), 367--385.

\bibitem{Wo}
R. Woodroofe, An EL-labeling of the subgroup lattice, {\it Proc. Amer. Math. Soc.} {\bf 136}
(2008), no. 11, 3795--3801.

\bibitem{Zaslavsky} T. Zaslavsky,  Facing up to Hyperplane Arrangements: Face-Count Formulas for Partitions of Space by Hyperplanes,   {\it Memoirs of the American Mathematical Society} \textbf{154}  (1975), 102 pp.

\end{thebibliography}


\end{document}